\documentclass{amsart}
\usepackage{amssymb}
\usepackage{amsfonts}

\newtheorem{acknowledgement}{Acknowledgement}
\newtheorem{theorem}{Theorem}[section]
\newtheorem{corollary}[theorem]{Corollary}
\newtheorem{lemma}[theorem]{Lemma}
\newtheorem{definition}[theorem]{Definition}
\newtheorem{proposition}[theorem]{Proposition}
\newtheorem{remark}[theorem]{Remark}
\newtheorem{example}[theorem]{Example}
\numberwithin{theorem}{section}
\input{tcilatex}

\begin{document}
\title[Scattered locally $C^{\ast }$-algebras]{Scattered locally $C^{\ast }$%
-algebras}
\author{Maria Joi\c{t}a}
\address{Maria Joi\c{t}a \\
Department of Mathematics, Faculty of Applied Sciences, University
Politehnica of Bucharest, 313 Spl. Independentei, 060042, Bucharest, Romania
and Simion Stoilow Institute of Mathematics of the Roumanian Academy, 21
Calea Grivitei Street, 010702 Bucharest, Romania}
\email{mjoita@fmi.unibuc.ro and maria.joita@mathem.pub.ro}
\urladdr{http://sites.google.com/a/g.unibuc.ro/maria-joita/ }
\thanks{This paper is in final form and no version of it will be submitted
for publication elsewhere.}

\begin{abstract}
In this paper, we introduce the notion of a scattered locally $C^{\ast }$%
-algebra and we give conditions for a locally $C^{\ast }$-algebra to be
scattered. Given an action $\alpha $ of a locally compact group $G$ on a
scattered locally $C^{\ast }$-algebra $A[\tau _{\Gamma }]$, it is natural to
ask under what conditions the crossed product $A[\tau _{\Gamma }]\times
_{\alpha }G$ is also scattered. We obtain some results concerning this
question.
\end{abstract}

\maketitle

\section{Introduction}

A topological space $X$ is called scattered (or dispersed) if every nonempty
subset of $X$ necessarily contains an isolated point. In \cite[p. 41,
Theorem 6]{R}, W. Rudin showed that the linear functionals on $C(X)$, where $%
X$ is a compact Hausdorff space which is scattered, have a very simple
structure. A compact Hausdorff space $X$ is scattered if and only if every
Radon measure on $X$ is atomic. A. Pelczynski and Z. Semadeni \cite{PS} gave
several necessary and sufficient conditions for a compact Hausdorff space $X$
to be scattered in terms of $C(X)$. They showed that a compact Hausdorff
space $X$ is scattered if and only if every linear functional $f$ on $C(X)$
is of the form 
\begin{equation*}
f(h)=\tsum\limits_{n=1}^{\infty }a_{n}h\left( x_{n}\right) 
\end{equation*}%
where $\left( x_{n}\right) _{n}$ is a fixed sequence of points in $X$ and $%
\tsum\limits_{n=1}^{\infty }\left\vert a_{n}\right\vert <\infty $. As a
non-commutative generalization of a scattered compact Hausdorff space, the
notion of a scattered $C^{\ast }$-algebra was introduced independently by H.
E. Jensen \cite{J1}\textbf{\ }and M. L. Rothwell \cite{Ro}. A $C^{\ast }$%
-algebra $A$ is said to be scattered if every positive functional on $A$ is
atomic \cite[Definition 1.1]{J1}; or equivalently, any positive functional
on $A$ is the sum of a finite or infinite sequence of pure functionals on $A$%
. The reader is referred to \cite{C}, \cite{H}, \cite{J1}, \cite{J2}, \cite%
{L}, \cite{K}, \cite{Ro} for other equivalent conditions on scattered $%
C^{\ast }$-algebras.

The notion of a locally $C^{\ast }$-algebra is a generalization of the
notion of a $C^{\ast }$-algebra, instead to be given by a single $C^{\ast }$%
-norm, the topology on a locally $C^{\ast }$-algebra is defined by a
directed family of $C^{\ast }$-seminorms. A locally $C^{\ast }$-algebra $%
A[\tau _{\Gamma }]$ is a complete Hausdorff topological $\ast $-algebra for
which there exists an upward directed family $\Gamma $ of $C^{\ast }$%
-seminorms $\{p_{\lambda }\}_{\lambda \in \Lambda }$ defining the topology $%
\tau _{\Gamma }$. A Fr\'{e}chet locally $C^{\ast }$-algebra is a locally $%
C^{\ast }$-algebra whose topology is given by a countably family of $C^{\ast
}$-seminorms. A morphism of locally $C^{\ast }$-algebras is a continuous $%
\ast $-morphism $\Phi $ from a locally $C^{\ast }$-algebra $A[\tau _{\Gamma
}]$ to another locally $C^{\ast }$-algebra $B[\tau _{\Gamma ^{\prime }}]$.
Other terms with which locally $C^{\ast }$-algebras can be found in the
literature are: pro-$C^{\ast }$-algebras (N.C. Phillips), $b^{\ast }$%
-algebras (C.\thinspace Apostol) and LMC$^{\ast }$-algebras (G.\thinspace
Lassner, K.\thinspace Schm{\"{u}}dgen).

Let $\{A_{\lambda };\chi _{\lambda \mu }\}_{\lambda ,\mu \in \Lambda
,\lambda \geq \mu }$ be an inverse system of $C^{\ast }$-algebras. Then $%
\lim\limits_{\leftarrow \lambda }A_{\lambda }$ with the topology given by
the family of $C^{\ast }$-seminorms $\{p_{A_{\lambda }}\}_{\lambda \in
\Lambda },p_{A_{\lambda }}\left( \left( a_{\mu }\right) _{\mu }\right)
=\left\Vert a_{\lambda }\right\Vert _{A_{\lambda }}$, where $\left\Vert
\cdot \right\Vert _{A_{\lambda }}$ denotes the $C^{\ast }$-norm on $%
A_{\lambda }$, is a locally $C^{\ast }$-algebra.

For a locally $C^{\ast }$-algebra $A[\tau _{\Gamma }]$, and every $\lambda
\in \Lambda $, the quotient normed $\ast $-algebra $A_{\lambda }=A/\ker
p_{\lambda }$, where $\ker p_{\lambda }=\{a\in A;p_{\lambda }(a)=0\}$, is
already complete, hence, it is a $C^{\ast }$-algebra in the norm $||a+\ker
p_{\lambda }||_{A_{\lambda }}$ $=p_{\lambda }(a),a\in A$ (C. Apostol, see
e.g. \cite[Theorem 10.24]{F}). The canonical map from $A$ to $A_{\lambda }$
is denoted by $\pi _{\lambda }^{A}$. For $\lambda ,\mu \in \Lambda $ with $%
\lambda \geq \mu $, there is a canonical surjective $C^{\ast }$-morphism $%
\pi _{\lambda \mu }^{A}:A_{\lambda }\rightarrow A_{\mu }$ such that $\pi
_{\lambda \mu }^{A}\left( a+\ker p_{\lambda }\right) =a+\ker p_{\mu }$ for
all $a\in A$. Moreover, $\{A_{\lambda };\pi _{\lambda \mu }^{A}\}_{\lambda
,\mu \in \Lambda ,\lambda \geq \mu }$ is an inverse system of $C^{\ast }$%
-algebras, called the Arens-Michael decomposition of the locally $C^{\ast }$%
-algebra $A[\tau _{\Gamma }]$. The Arens-Michael decomposition gives us a
representation of $A[\tau _{\Gamma }]$ as an inverse limit of $C^{\ast }$%
-algebras; namely $A[\tau _{\Gamma }]=\lim\limits_{\leftarrow \lambda
}A_{\lambda }$, up to a topological $\ast $-isomorphism.

In this paper, we introduce the notion of scattered locally $C^{\ast }$%
-algebra and we give conditions for locally $C^{\ast }$-algebras to be
scattered. Given an action $\alpha $ of a locally compact group $G$ on a
scattered locally $C^{\ast }$-algebra $A[\tau _{\Gamma }]$, it is natural to
ask under what condition the crossed product $A[\tau _{\Gamma }]\times
_{\alpha }G$ is also scattered. We obtain some results concerning this
question.

\section{Scattered locally $C^{\ast }$-algebras}

Let $A[\tau _{\Gamma }]$ be a locally $C^{\ast }$-algebra. A continuous
positive functional on $A[\tau _{\Gamma }]\ $is a continuous linear map $%
f:A\rightarrow \mathbb{C}$ with the property that $f\left( a^{\ast }a\right)
\geq 0$ for all $a\in A$. If $f_{\lambda }$ is a positive functional on $%
A_{\lambda }$, then $f_{\lambda }\circ \pi _{\lambda }^{A}$ is a continuous
positive functional on $A[\tau _{\Gamma }]$. Moreover, for any continuous
positive functional $f$ on $A[\tau _{\Gamma }]$, there are $\lambda \in
\Lambda $ and a positive functional $f_{\lambda }$ on $A_{\lambda }$,$\ $%
called the positive functional associated to $f$, such that $f=f_{\lambda
}\circ \pi _{\lambda }^{A}$. A continuous positive functional $f$ on $A[\tau
_{\Gamma }]\ $is pure if $f\neq 0$ and if $g$ is another positive functional
on $A[\tau _{\Gamma }]$ and $g\leq f$, then there is $\alpha \in \lbrack
0,1] $ such that $g=\alpha f$. A continuous positive functional $f$ on $%
A[\tau _{\Gamma }]\ $is pure if and only if its associated positive
functional $f_{\lambda }$ is pure.

\begin{definition}
A locally $C^{\ast }$-algebra $A[\tau _{\Gamma }]$ is scattered if any
continuous positive functional $f$ on $A[\tau _{\Gamma }]$ is a countable
sum $f=\tsum\limits_{n}f_{n}$ of pure functionals $f_{n}$ on $A$, in the
pointwise convergence.
\end{definition}

\begin{remark}
Let $A[\tau _{\Gamma }]$ and $B[\tau _{\Gamma ^{\prime }}]\ $be two
isomorphic locally $C^{\ast }$-algebras. Then $A[\tau _{\Gamma }]$ is
scattered if and only if $B[\tau _{\Gamma ^{\prime }}]$ is scattered.
\end{remark}

\begin{proposition}
\label{scatt} Let $A[\tau _{\Gamma }]$ be a locally $C^{\ast }$-algebra.
Then $A[\tau _{\Gamma }]$ is scattered if and only if the factors $%
A_{\lambda },\lambda \in \Lambda $ in the Arens-Michael decomposition of $%
A[\tau _{\Gamma }],$ are scattered.
\end{proposition}

\begin{proof}
First we suppose that $A[\tau _{\Gamma }]$ is scattered. Let $\lambda \in
\Lambda $ and let $f$ $\ $be a positive functional on $A_{\lambda }$. Then $%
f\circ \pi _{\lambda }^{A}$ is a continuous positive functional on $A[\tau
_{\Gamma }]$, and since $A[\tau _{\Gamma }]$ is scattered, $f\circ \pi
_{\lambda }^{A}=\tsum\limits_{n}f_{n}$, where $f_{n},$ $n\in \mathbb{N},$
are pure. Let $n$ be a positive integer. Then, there are $\mu \in \Lambda $
and $f_{\mu ,n}$ a positive functional on $A_{\mu }$ such that $f_{n}=f_{\mu
,n}$ $\circ \pi _{\mu }^{A}$. Moreover,%
\begin{equation*}
\left\vert f_{n}\left( a\right) \right\vert ^{2}\leq \left\Vert f_{\mu
,n}\right\Vert f_{n}\left( a^{\ast }a\right) \leq \left\Vert f_{\mu
,n}\right\Vert f\left( \pi _{\lambda }^{A}\left( a^{\ast }a\right) \right)
\leq \left\Vert f_{\mu ,n}\right\Vert \left\Vert f\right\Vert p_{\lambda
}\left( a\right) ^{2}
\end{equation*}%
for all $a\in A[\tau _{\Gamma }]$.$\ $Therefore, for each positive integer $%
n $, there is a positive functional $f_{n}^{\lambda }$ on $A_{\lambda }$
such that $f_{n}^{\lambda }\circ \pi _{\lambda }^{A}=f_{n}$ . Moreover,
since $f_{n}$ is pure, \ $f_{n}^{\lambda }$ is pure and $f=\tsum%
\limits_{n}f_{n}^{\lambda }$.

Conversely, suppose that $f\ $is a continuous positive functional on $A[\tau
_{\Gamma }]$. Then there are $\lambda \in \Lambda $ and a positive
functional $f_{\lambda }$ on $A_{\lambda }$ such that $f=f_{\lambda }\circ
\pi _{\lambda }^{A}$. Since $A_{\lambda }\ $is scattered,\ $f_{\lambda
}=\tsum\limits_{n}f_{n}$, where $f_{n}$ are pure. Then 
\begin{equation*}
f=f_{\lambda }\circ \pi _{\lambda }^{A}=\tsum\limits_{n}f_{n}\circ \pi
_{\lambda }^{A}
\end{equation*}%
and since for each positive integer $n,$ $f_{n}\circ \pi _{\lambda }^{A}$ is
pure, $A[\tau _{\Gamma }]$ is scattered.
\end{proof}

\begin{corollary}
\label{sub} Any closed $\ast $- subalgebra of a scattered locally $C^{\ast }$%
-algebra is a scattered locally $C^{\ast }$-algebra.
\end{corollary}

\begin{proof}
Let $A[\tau _{\Gamma }]$ be a scattered locally $C^{\ast }$-algebra and $B$
a closed $\ast $- subalgebra of $A[\tau _{\Gamma }]$. Then $B$ is a locally $%
C^{\ast }$-algebra and the factors $B_{\lambda },\lambda \in \Lambda ,$ in
the Arens-Michael decomposition of $B$, can be identified with the $C^{\ast
} $-subalgebras $\overline{\pi _{\lambda }^{A}(B)}$, the closure of the $%
\ast $-subalgebra $\pi _{\lambda }^{A}(B)$ in $A_{\lambda }$, of$\
A_{\lambda },\lambda \in \Lambda $ which are scattered $C^{\ast }$-algebras.
Then, $B_{\lambda },\lambda \in \Lambda $, are scattered (see, for example, 
\cite[p. 677]{K1}) and so $B$ is scattered.
\end{proof}

A Hausdorff countably compactly generated topological space is a topological
space $X$ which is the direct limit of a sequence of Hausdorff compact
spaces $\{K_{n}\}_{n}$. The $\ast $-algebra $C(X)$ of all continuous complex
valued functions on $X$ has a structure of a locally $C^{\ast }$-algebra
with respect to the topology given by $C^{\ast }$-seminorms $%
\{p_{K_{n}}\}_{n}$ with $p_{K_{n}}\left( f\right) =\sup \{\left\vert f\left(
x\right) \right\vert ;x\in K_{n}\}$. Moreover, for each $n,C(X)_{n}\ $is
isomorphic to $C(K_{n})$, and for any commutative Fr\'{e}chet locally $%
C^{\ast }$-algebra $A$ there is a Hausdorff countably compactly generated
topological space $X$ such that $A$ is isomorphic with $C(X)$ \cite[Theorem
5.7]{P}.

\begin{corollary}
A commutative Fr\'{e}chet locally $C^{\ast }$-algebra $A[\tau _{\Gamma }]$
is scattered if and only if there is a Hausdorff countably compactly
generated topological space $X$ which is the direct limit of a sequence of
scattered Hausdorff compact spaces $\{K_{n}\}_{n}\ $such that $A[\tau
_{\Gamma }]$ is isomorphic with $C(X).$
\end{corollary}

\begin{corollary}
\label{tens}Let $A[\tau _{\Gamma }]$ and $B[\tau _{\Gamma ^{\prime }}]\ $be
two locally $C^{\ast }$-algebras. Then the maximal tensor product $A[\tau
_{\Gamma }]$ $\otimes _{\max }B[\tau _{\Gamma ^{\prime }}]$ of $A[\tau
_{\Gamma }]$ and $B[\tau _{\Gamma ^{\prime }}]$ is a scattered locally $%
C^{\ast }$-algebra if and only if the locally $C^{\ast }$-algebras $A[\tau
_{\Gamma }]$ and $B[\tau _{\Gamma ^{\prime }}]\ $are scattered.
\end{corollary}

\begin{proof}
It follows from Proposition \ref{scatt}, \cite[Proposition 1]{C} and \cite[%
Theorem 31.7 and Corollary 31.11]{F}.
\end{proof}

Recall that a continuous $\ast $-representation of a locally $C^{\ast }$%
-algebra $A[\tau _{\Gamma }]$ on a Hilbert space is a pair $\left( \varphi ,%
\mathcal{H}_{\varphi }\right) $ consisting of a Hilbert space $\mathcal{H}%
_{\varphi }$ and a continuous $\ast $-morphism $\varphi $ from $A[\tau
_{\Gamma }]$ to $L(\mathcal{H}_{\varphi })$, the $C^{\ast }$-algebra of all
bounded linear operators on $\mathcal{H}_{\varphi }$. If $\left( \varphi ,%
\mathcal{H}_{\varphi }\right) $ is a representation of $A[\tau _{\Gamma }]$,
then there exist $\lambda \in \Lambda $, and a $\ast $-representation $%
\left( \varphi _{\lambda },\mathcal{H}_{\varphi }\right) $ of $A_{\lambda }$
such that $\varphi =\varphi _{\lambda }\circ \pi _{\lambda }^{A}$.

A continuous $\ast $-representation $\left( \varphi ,\mathcal{H}_{\varphi
}\right) \ $of $A[\tau _{\Gamma }]$ is of type $I$ if the von Neumann
algebra generated by $\varphi \left( A\right) $ is of type $I$ (that is, the
commutant of $\varphi \left( A\right) $ is an abelian $\ast $-subalgebra of $%
L(\mathcal{H}_{\varphi })$ ). A locally $C^{\ast }$-algebra $A[\tau _{\Gamma
}]$ is of type $I$ if each of its continuous $\ast $-representations is of
type $I$.

\begin{corollary}
\label{type}Let $A[\tau _{\Gamma }]$ be a locally $C^{\ast }$-algebra. If $%
A[\tau _{\Gamma }]$ is scattered, then $A[\tau _{\Gamma }]$ is of type $I$.
\end{corollary}

\begin{proof}
Since $A[\tau _{\Gamma }]$ is scattered, the $A_{\lambda }$'s are scattered,
as well and by \cite[Theorem 2.3]{J1}, $A_{\lambda },\lambda \in \Lambda $
are of type $I$. The corollary is proved, since $A[\tau _{\Gamma }]$ is of
type $I$ if and only if $A_{\lambda },\lambda \in \Lambda ,$ are of type $I$ 
\cite[Proposition 30.8]{F}.
\end{proof}

Let $\mathcal{I}$ be a closed two-sided $\ast $-ideal of $A[\tau _{\Gamma }]$%
. Then the quotient $\ast $-algebra $A[\tau _{\Gamma }]/\mathcal{I}$,
equipped with the quotient topology, is a pre-locally $C^{\ast }$-algebra,
and its completion $\overline{A[\tau _{\Gamma }]/\mathcal{I}}$ is a locally $%
C^{\ast }$-algebra. Moreover, for each $\lambda \in \Lambda ,$ $\mathcal{I}%
_{\lambda }=\overline{\pi _{\lambda }^{A}(\mathcal{I})}$, the closure of $%
\pi _{\lambda }^{A}(\mathcal{I})$ in the\textbf{\ }$C^{\ast }$-algebra%
\textbf{\ }$A_{\lambda },$ is a closed two-sided $\ast $-ideal of $%
A_{\lambda }$ and the $C^{\ast }$-algebras $A_{\lambda }/\mathcal{I}%
_{\lambda }$ and $\left( \overline{A[\tau _{\Gamma }]/\mathcal{I}}\right)
_{\lambda }$ are isomorphic (see, for example, \cite[Theorem 11.7]{F}).

\begin{proposition}
Let $A[\tau _{\Gamma }]$ be a locally $C^{\ast }$-algebra and $\mathcal{I}$
a closed two-sided $\ast $-ideal of $A[\tau _{\Gamma }]$. Then $A[\tau
_{\Gamma }]$ is scattered if and only if $\mathcal{I}$ and $\overline{A[\tau
_{\Gamma }]/\mathcal{I}}$ are scattered.
\end{proposition}

\begin{proof}
It follows from the above discussion, \cite[Proposition 2.4]{J1} and
Proposition \ref{scatt}.
\end{proof}

\begin{remark}
If $A[\tau _{\Gamma }]$ is a Fr\'{e}chet locally $C^{\ast }$-algebra and $%
\mathcal{I}$ is a closed two-sided $\ast $-ideal of $A[\tau _{\Gamma }]$,
then the quotient $\ast $-algebra $A[\tau _{\Gamma }]/\mathcal{I}$ is
complete and so it is a Fr\'{e}chet locally $C^{\ast }$-algebra (see, for
example, \cite[Corollary 11.8]{F}). Therefore, if $A[\tau _{\Gamma }]$ is a
Fr\'{e}chet locally $C^{\ast }$-algebra and $\mathcal{I}$ is a closed
two-sided $\ast $-ideal of $A[\tau _{\Gamma }]$, then $A[\tau _{\Gamma }]$
is scattered if and only if $\mathcal{I}$ and $A[\tau _{\Gamma }]/\mathcal{I}
$ are scattered.
\end{remark}

An element $a$ in a locally $C^{\ast }$-algebra $A[\tau _{\Gamma }]$ is
bounded if $\sup \{p_{\lambda }\left( a\right) ;\lambda \in \Lambda
\}<\infty $. Put $b\left( A[\tau _{\Gamma }]\right) =\{a\in A[\tau _{\Gamma
}];a$ is bounded\}. The map $\left\Vert \cdot \right\Vert _{\infty }:$ $%
b\left( A[\tau _{\Gamma }]\right) \rightarrow \lbrack 0,\infty )$ defined by 
\begin{equation*}
\left\Vert a\right\Vert _{\infty }=\sup \{p_{\lambda }\left( a\right)
;\lambda \in \Lambda \}
\end{equation*}%
is a $C^{\ast }$-norm, and $b\left( A[\tau _{\Gamma }]\right) $, equipped
with this $C^{\ast }$-norm, is a $C^{\ast }$-algebra which is dense in $%
A[\tau _{\Gamma }]$. Moreover, for each $\lambda \in \Lambda ,$ $\ker
p_{\lambda }|_{b(A[\tau _{\Gamma }])}=\ker p_{\lambda }\cap b\left( A[\tau
_{\Gamma }]\right) $ is a closed two-sided $\ast $-ideal of $b\left( A[\tau
_{\Gamma }]\right) $ and the $C^{\ast }$-algebras $b\left( A[\tau _{\Gamma
}]\right) /\ker p_{\lambda }|_{b(A[\tau _{\Gamma }])}$ and $A_{\lambda }$
are isomorphic (see, for example, \cite[Theorem 10.24]{F} ).

\begin{proposition}
\label{bond}Let $A[\tau _{\Gamma }]$ be a locally $C^{\ast }$-algebra.

\begin{enumerate}
\item If the $C^{\ast }$-algebra $b(A[\tau _{\Gamma }])$ of all bounded
elements is scattered, then $A[\tau _{\Gamma }]$ is scattered.

\item If $A[\tau _{\Gamma }]$\ is scattered and for some $\lambda \in
\Lambda $, the closed two sided $\ast $-ideal $\ker p_{\lambda }|_{b(A[\tau
_{\Gamma }])}$ of $b(A[\tau _{\Gamma }])$ is scattered, then $b(A[\tau
_{\Gamma }])$\ is scattered.
\end{enumerate}
\end{proposition}

\begin{proof}
$(1)$ If $b(A[\tau _{\Gamma }])$ is scattered, then for each $\lambda \in
\Lambda ,$ $\ker p_{\lambda }|_{b(A[\tau _{\Gamma }])}$ and $b(A[\tau
_{\Gamma }])/$ $\ker p_{\lambda }|_{b(A[\tau _{\Gamma }])}$ are scattered 
\cite[Proposition 2.4]{J1}. Therefore, for each $\lambda \in \Lambda $, the $%
C^{\ast }$-algebra $A_{\lambda }$ is scattered, since $A_{\lambda }$ is
isomorphic with $b(A[\tau _{\Gamma }])/\ker p_{\lambda }|_{b(A[\tau _{\Gamma
}])}$, and by Proposition \ref{scatt}, $A[\tau _{\Gamma }]$ is scattered.

$(2)\ $If $A[\tau _{\Gamma }]$\ is scattered, then $A_{\lambda }$ is
scattered, and since $A_{\lambda }\ $is isomorphic with $b(A[\tau _{\Gamma
}])/$ $\ker p_{\lambda }|_{b(A[\tau _{\Gamma }])}$ and the closed two sided $%
\ast $-ideal $\ker p_{\lambda }|_{b(A[\tau _{\Gamma }])}$ of $b(A[\tau
_{\Gamma }])$ is scattered, $b(A[\tau _{\Gamma }])$\ is scattered.
\end{proof}

Let $A[\tau _{\Gamma }]$ be a locally $C^{\ast }$-algebra and $Z(A[\tau
_{\Gamma }])=\{a\in A;ab=ba$ for all $b\in A\}$ its center. Clearly, $%
Z(A[\tau _{\Gamma }])$ is a commutative locally $C^{\ast }$-subalgebra of $A$%
, and so it is a locally $C^{\ast }$-algebra with respect to the topology
given by the family of $C^{\ast }$-seminorms $\{p_{\lambda }|_{Z(A[\tau
_{\Gamma }])}\}_{\lambda \in \Lambda }$. For each $\lambda ,\mu \in \Lambda $
with $\lambda \geq \mu ,\pi _{\lambda \mu }^{A}\left( Z(A_{\lambda })\right)
\subseteq Z(A_{\mu })$ and so $\{Z(A_{\lambda });\pi _{\lambda \mu
}^{A}|_{Z(A_{\lambda })}\}_{\lambda ,\mu \in \Lambda ,\lambda \geq \mu }$ is
an inverse system of $C^{\ast }$-algebras.

\begin{proposition}
Let $A[\tau _{\Gamma }]$ be a locally $C^{\ast }$-algebra. Then $Z(A[\tau
_{\Gamma }])=\lim\limits_{\leftharpoonup \lambda }Z(A_{\lambda })$, up to an
isomorphism of locally $C^{\ast }$-algebras.
\end{proposition}

\begin{proof}
Consider the map $\Phi :Z(A[\tau _{\Gamma }])\rightarrow
\lim\limits_{\leftharpoonup \lambda }Z(A_{\lambda })\ $given by 
\begin{equation*}
\Phi \left( a\right) =\left( \pi _{\lambda }^{A}\left( a\right) \right)
_{\lambda }\text{.}
\end{equation*}%
Clearly,$\ \Phi $ is a $\ast $-morphism$\ $and $p_{Z(A_{\lambda })}\left(
\Phi \left( a\right) \right) =p_{\lambda }|_{Z(A[\tau _{\Gamma }])}\left(
a\right) $ for all $a\in Z(A[\tau _{\Gamma }])$ and for all $\lambda \in
\Lambda $. If $\left( a_{\lambda }\right) _{\lambda }$is a coherent sequence
in $\{Z(A_{\lambda });\pi _{\lambda \mu }^{A}|_{Z(A_{\lambda })}\}_{\lambda
,\mu \in \Lambda ,\lambda \geq \mu }$, then there is $a\in A$ such that $\pi
_{\lambda }^{A}\left( a\right) =a_{\lambda }$ for all $\lambda \in \Lambda $%
. Take $b\in A$. From $\pi _{\lambda }^{A}\left( ab\right) =\pi _{\lambda
}^{A}\left( a\right) \pi _{\lambda }^{A}\left( b\right) =\pi _{\lambda
}^{A}\left( b\right) \pi _{\lambda }^{A}\left( a\right) =\pi _{\lambda
}^{A}\left( ba\right) $ for all $\lambda \in \Lambda $, we deduce that $ab=ba
$, and so $a\in Z(A[\tau _{\Gamma }])$. Therefore, $\Phi $ is an isomorphism
of locally $C^{\ast }$-algebras. 
\end{proof}

\begin{remark}
We remark that, in general, the isometric $C^{\ast }$-morphism $\varphi
_{\lambda }:\left( Z(A[\tau _{\Gamma }])\right) _{\lambda }$ $\rightarrow
Z(A_{\lambda }),\varphi _{\lambda }\left( a+\ker \left( p_{\lambda
}|_{Z(A[\tau _{\Gamma }])}\right) \right) =a+\ker p_{\lambda }$ is not onto.
\end{remark}

An inverse system $\{A_{i};\chi _{ij}\}_{i,j\in I,i\geq j}$ of topological
algebras is called \textit{perfect}, if the restrictions to the inverse
limit algebra $A=\lim\limits_{\leftarrow i}A_{i}$ of the canonical
projections $\pi _{i}:\tprod\limits_{i\in I}A_{i}\rightarrow A_{i},i\in I$,
namely, the continuous morphisms $\pi _{i}|_{A}:A\rightarrow A_{i},i\in I$,
are onto maps. The resulted inverse limit algebra $A=\lim\limits_{\leftarrow
i}A_{i}$ is called \textit{perfect topological algebra} \cite[Definition 2.7]%
{HM}.

\begin{definition}
We say that a locally $C^{\ast }$-algebra $A[\tau _{\Gamma }]\ $is with
perfect center, if the inverse system of $C^{\ast }$-algebras $%
\{Z(A_{\lambda });\pi _{\lambda \mu }^{A}|_{Z(A_{\lambda })}\}_{\lambda ,\mu
\in \Lambda ,\lambda \geq \mu }$ is perfect.
\end{definition}

Let $\{H_{\lambda }\}_{\lambda \in \Lambda }$ be a directed family of
Hilbert spaces such that for each $\lambda ,\mu \in \Lambda $ with $\lambda
\geq \mu ,$ $H_{\mu }$ is a closed subspace of $H_{\lambda }$ and $%
\left\langle \cdot ,\cdot \right\rangle _{\mu }=\left\langle \cdot ,\cdot
\right\rangle _{\lambda }|_{H_{\mu }}$. Then $H=\lim\limits_{\lambda
\rightarrow }H_{\lambda }$ with the inductive limit topology is called a 
\textit{locally Hilbert space.} $L(H)$ denotes all linear maps $%
T:H\rightarrow H$ such that for each $\lambda \in \Lambda $, $T|_{H_{\lambda
}}\in L(H_{\lambda })$,$\ $the $C^{\ast }$-algebra of all bounded linear
operators on $H_{\lambda }$, and $P_{\lambda \mu }T|_{H_{\lambda
}}=T|_{H_{\lambda }}P_{\lambda \mu }$ for all $\lambda ,\mu \in \Lambda $
with $\lambda \geq \mu $, where $P_{\lambda \mu }$ is the projection of $%
H_{\lambda }$ on $H_{\mu }$. If $T\in L(H)$, then there is $T^{\ast }\in
L(H)\ $such that $T^{\ast }|_{H_{\lambda }}=\left( T|_{H_{\lambda }}\right)
^{\ast }$ for all $\lambda \in \Lambda $. Then $L(H)$ has a structure of
locally $C^{\ast }$-algebra with the topology given by the family of $%
C^{\ast }$-seminorms $\{p_{\lambda }\}_{\lambda \in \Lambda }$ with $%
p_{\lambda }\left( T\right) =\left\Vert T|_{H_{\lambda }}\right\Vert
_{L(H_{\lambda })}\ $(see, for example, \cite[Theorem 5.1]{I}).

\begin{example}
Let $H=\lim\limits_{\lambda \rightarrow }H_{\lambda }$ be a locally Hilbert
space. Then $H$ is a pre-Hilbert space with the inner product given by $%
\left\langle \xi ,\eta \right\rangle =\left\langle \xi ,\eta \right\rangle
_{\lambda }$ if $\xi ,\eta \in H_{\lambda }$. Let $\widetilde{H}\ $be the
Hilbert space obtained by the completion of $H$. For each $\lambda \in
\Lambda ,$ $H_{\lambda }$ is a closed subspace of $\widetilde{H}$.\ The
projection of $\widetilde{H}$ \ on $H_{\lambda }$ is denoted by $P_{\lambda
} $. Clearly, the restriction $P_{\lambda }|_{H}\ $of $P_{\lambda }\ $on$\ H$
is an element in $L(H)$. It is easy to check that $Z(L(H))$ is the locally $%
C^{\ast }$-subalgebra of $L(H)$ generated by $\{P_{\lambda }|_{H},\lambda
\in \Lambda \}$ and for each $\lambda \in \Lambda ,$ $\left( Z(L(H))\right)
_{\lambda }$ is isomorphic with the $C^{\ast }$-subalgebra of $L(H_{\lambda
})\ $generated\ by $\{P_{\lambda }|_{H_{\mu }},\mu \in \Lambda ,\mu \leq
\lambda \}$.

On\ the other hand, $L(H)_{\lambda }\ $is isomorphic with the $C^{\ast }$%
-subalgebra of $L(H_{\lambda })$, the $C^{\ast }$-algebra of all bounded
linear operators on $H_{\lambda }$, generated by $\{T\in L(H_{\lambda
});P_{\lambda \mu }T=TP_{\lambda \mu },\mu \in \Lambda ,\mu \leq \lambda \}\ 
$and then $Z(L(H)_{\lambda })$ is isomorphic with the $C^{\ast }$-subalgebra
of $L(H_{\lambda })\ $generated\ by $\{P_{\lambda }|_{H_{\mu }},\mu \in
\Lambda ,\mu \leq \lambda \}$.

Therefore, for each $\lambda \in \Lambda ,$ the $C^{\ast }$-algebras $\left(
Z(L(H))\right) _{\lambda }$ and $Z(L(H)_{\lambda })$ are isomorphic and $%
L(H) $ is a locally $C^{\ast }$-algebra with perfect center.

If the Hilbert spaces $H_{\lambda },\lambda \in \Lambda ,$ are finite
dimensional, then the $C^{\ast }$-algebras $L(H_{\lambda }),\lambda \in
\Lambda $, are scattered \cite{C}. Therefore the factors $L(H)_{\lambda
},\lambda \in \Lambda $, in the the Arens-Michael decomposition of $L(H)$
are scattered, and, by Proposition \ref{scatt}, $L(H)$ is a scattered
locally $C^{\ast }$-algebra whit perfect center.
\end{example}

It is know that a $C^{\ast }$-algebra $A$ is scattered if and only if it is
of type $I$ and its center $Z(A)$ is a scattered $C^{\ast }$-algebra \cite[%
Theorem 2.2]{K}. The following result is a generalization of \cite[Theorem
2.2]{K}.

\begin{theorem}
Let $A[\tau _{\Gamma }]$ be a locally $C^{\ast }$-algebra with perfect
center. Then the following statements are equivalent:

\begin{enumerate}
\item $A[\tau _{\Gamma }]$ is a scattered locally $C^{\ast }$-algebra;

\item $A[\tau _{\Gamma }]$ is of type $I$ and $Z(A[\tau _{\Gamma }])$ is a
scattered locally $C^{\ast }$-algebra.
\end{enumerate}
\end{theorem}

\begin{proof}
$\left( 1\right) \Rightarrow \left( 2\right) .$ It follows from Corollaries %
\ref{sub} and \ref{type}.

$\left( 2\right) \Rightarrow \left( 1\right) .$ Since $Z(A[\tau _{\Gamma }])$
is a scattered locally $C^{\ast }$-algebra, the factors $\left( Z(A[\tau
_{\Gamma }])\right) _{\lambda },$ $\lambda \in \Lambda ,\ $in the
Arens-Michael decomposition of $Z(A[\tau _{\Gamma }])$ are scattered. On the
other hand, $Z(A[\tau _{\Gamma }])=\lim\limits_{\leftharpoonup \lambda
}Z(A_{\lambda })$, up to an isomorphism of locally $C^{\ast }$-algebras, and 
$\{Z(A_{\lambda });\pi _{\lambda \mu }^{A}|_{Z(A_{\lambda })}\}_{\lambda
,\mu \in \Lambda ,\lambda \geq \mu }$ is a perfect inverse system of $%
C^{\ast }$-algebras. Therefore, the $C^{\ast }$-algebras $Z\left( A_{\lambda
}\right) ,\lambda \in \Lambda ,$ are scattered. Since $A[\tau _{\Gamma }]$
is of type $I,$ the factors $A_{\lambda },\lambda \in \Lambda ,\ $in the
Arens-Michael decomposition of $A[\tau _{\Gamma }]$ are of type $I$. Then,
by \cite[Theorem 2.2]{K}, the $C^{\ast }$-algebras $A_{\lambda },\lambda \in
\Lambda ,$ are scattered, and by Proposition \ref{scatt}, $A[\tau _{\Gamma
}] $ is scattered.
\end{proof}

\section{Crossed products of scattered locally $C^{\ast }$-algebras}

Let $G$ be a locally compact group and let $A[\tau _{\Gamma }]$ be a locally 
$C^{\ast }$-algebra. An action of $G\ $on $A[\tau _{\Gamma }]$ is a group
morphism $\alpha $ from $G$ to Aut$\left( A[\tau _{\Gamma }]\right) $, the
group of all automorphisms of $A[\tau _{\Gamma }]$, such that, for each $%
a\in A$, the map $g\mapsto \alpha _{g}(a)\ $from $G$ to $A[\tau _{\Gamma }]$
is continuous. An action $\alpha $ of $G$ on $A[\tau _{\Gamma }]$ is an
inverse limit action if there is a cofinal subset $\Gamma ^{\prime }$ of $%
\Gamma $ with the property that $p_{\lambda }\left( \alpha _{g}\left(
a\right) \right) =p_{\lambda }\left( a\right) $ for all $a\in A$, for all $%
g\in G$ and for all $p_{\lambda }\in \Gamma ^{\prime }$. If $\alpha \ $is an
inverse limit action, we can suppose that $\Gamma ^{\prime }$ $=\Gamma $,
and then for each $\lambda \in \Lambda $, there is an action $\alpha
^{\lambda }$ of $G$ on $A_{\lambda }$ such that $\alpha
_{g}=\lim\limits_{\leftarrow \lambda }\alpha _{g}^{\lambda }$ for all $g\in
G $. If $G\ $is compact, then any action of $G$ on $A[\tau _{\Gamma }]$ is
an inverse limit action.

Recall that if $\alpha \ $is an inverse limit action of $G$ on $A[\tau
_{\Gamma }]$, then $L^{1}(G,\alpha ,A[\tau _{\Gamma }])$ $=\{f:G\rightarrow
A;\tint\limits_{G}p_{\lambda }\left( f(g)\right) dg<\infty $\ for all $%
\lambda \in \Lambda \}$, where $dg$ is the Haar measure on $G$, has a
structure of locally $m$-convex $\ast $-algebra with the convolution as
product and involution given by $f^{\#}(g)=\Delta (g^{-1})$ $\alpha
_{g}\left( f(g^{-1})^{\ast }\right) $, where $\Delta $ is the modular
function on $G$, and the topology given by the family of submultiplicative $%
\ast $-seminorms $\{N_{\lambda }\}_{\lambda }$, where $N_{\lambda
}(f)=\tint\limits_{G}p_{\lambda }\left( f(g)\right) dg.$ The crossed product
of $A[\tau _{\Gamma }]$ by $\alpha $, denoted by $A[\tau _{\Gamma }]$ $%
\times _{\alpha }G$, is the enveloping locally $C^{\ast }$-algebra of the
covariant algebra $L^{1}(G,\alpha ,A[\tau _{\Gamma }])$. Moreover, for each $%
\lambda \in \Lambda $, the $C^{\ast }$-algebras $\left( A[\tau _{\Gamma
}]\times _{\alpha }G\right) _{\lambda }$ and $A_{\lambda }$ $\times _{\alpha
^{\lambda }}G$ are isomorphic (see, \cite{J}).

As in the case of $C^{\ast }$-algebras, we have the following result.

\begin{proposition}
Let $G$ be a locally compact group and $A[\tau _{\Gamma }]$ a locally $%
C^{\ast }$-algebra. Then the crossed product $A[\tau _{\Gamma }]\times
_{\iota }G$ of $A[\tau _{\Gamma }]$ by the trivial action $\iota $ of $G$ is
scattered if and only if $A[\tau _{\Gamma }]$ and $C^{\ast }\left( G\right) $%
, the group $C^{\ast }$-algebra associated to $G$, are scattered.
\end{proposition}

\begin{proof}
The assertion follows from Corollary \ref{tens}, by taking into account that 
$A[\tau _{\Gamma }]\times _{\iota }G$ is isomorphic to the maximal tensor
product of $A[\tau _{\Gamma }]$ and $C^{\ast }\left( G\right) \ $(see, for
example, \cite{J}).
\end{proof}

The following result extends \cite[Proposition 6]{C}.

\begin{proposition}
Let $G$ be a compact group and $\alpha $ an action of $G$ on a locally $%
C^{\ast }$-algebra $A[\tau _{\Gamma }]$. If $A[\tau _{\Gamma }]$ is
scattered, then $A[\tau _{\Gamma }]$ $\times _{\alpha }G$ is scattered.
\end{proposition}

\begin{proof}
If $A[\tau _{\Gamma }]$ is scattered, then, for each $\lambda \in \Lambda
,A_{\lambda }$ is scattered (Proposition \ref{scatt}), and by \cite[%
Proposition 6]{C} $A_{\lambda }$ $\times _{\alpha ^{\lambda }}G$ is
scattered. From these facts and taking into account that for each $\lambda
\in \Lambda $, the $C^{\ast }$-algebras $\left( A[\tau _{\Gamma }]\times
_{\alpha }G\right) _{\lambda }$ and $A_{\lambda }$ $\times _{\alpha
^{\lambda }}G$ are isomorphic, we deduce that $A[\tau _{\Gamma }]$ $\times
_{\alpha }G$ is scattered.
\end{proof}

Let $\alpha =\lim\limits_{\leftarrow \lambda }\alpha ^{\lambda }$ be an
inverse limit action of a locally compact group $G$ on a pro-$C^{\ast }$%
-algebra $A[\tau _{\Gamma }]$ and $\left( A[\tau _{\Gamma }]\right) ^{\alpha
}=\{a\in A;\alpha _{g}\left( a\right) =a$ for all $g\in G\}\ $the fixed
point algebra of $A[\tau _{\Gamma }]$ under $\alpha $. Then $\left( A[\tau
_{\Gamma }]\right) ^{\alpha }$ is a locally $C^{\ast }$-subalgebra of $%
A[\tau _{\Gamma }]$. Since, for each $\lambda ,\mu \in \Lambda $ with $%
\lambda \geq \mu ,$ $\pi _{\lambda \mu }^{A}\left( \left( A_{\lambda
}\right) ^{\alpha ^{\lambda }}\right) $ $\subseteq \left( A_{\mu }\right)
^{\alpha ^{\mu }}$, 
\begin{equation*}
\{\left( A_{\lambda }\right) ^{\alpha ^{\lambda }};\pi _{\lambda \mu
}^{A}|_{\left( A_{\lambda }\right) ^{\alpha ^{\lambda }}}\}_{\lambda ,\mu
\in \Lambda ,\lambda \geq \mu }
\end{equation*}%
is an inverse system of $C^{\ast }$-algebras.

\begin{proposition}
Let $\alpha =\lim\limits_{\leftarrow \lambda }\alpha ^{\lambda }$ be an
inverse limit action of a locally compact group $G$ on a locally $C^{\ast }$%
-algebra $A[\tau _{\Gamma }]$. Then $\left( A[\tau _{\Gamma }]\right)
^{\alpha }=\lim\limits_{\leftharpoonup \lambda }\left( A_{\lambda }\right)
^{\alpha ^{\lambda }}$, up to an isomorphism of locally $C^{\ast }$-algebras.
\end{proposition}

\begin{proof}
Consider the map $\Psi :\left( A[\tau _{\Gamma }]\right) ^{\alpha
}\rightarrow \lim\limits_{\leftharpoonup \lambda }\left( A_{\lambda }\right)
^{\alpha ^{\lambda }}$given by 
\begin{equation*}
\Phi \left( a\right) =\left( \pi _{\lambda }^{A}\left( a\right) \right)
_{\lambda }\text{.}
\end{equation*}%
Clearly,$\ \Psi $ is a $\ast $-morphism$\ $and $p_{\left( A_{\lambda
}\right) ^{\alpha ^{\lambda }}}\left( \Psi \left( a\right) \right)
=p_{\lambda }|_{\left( A[\tau _{\Gamma }]\right) ^{\alpha }}\left( a\right) $
for all $a\in \left( A[\tau _{\Gamma }]\right) ^{\alpha }$ and for all $%
\lambda \in \Lambda $. If $\left( a_{\lambda }\right) _{\lambda }$is a
coherent sequence in $\{\left( A_{\lambda }\right) ^{\alpha ^{\lambda }};\pi
_{\lambda \mu }^{A}|_{\left( A_{\lambda }\right) ^{\alpha ^{\lambda
}}}\}_{\lambda ,\mu \in \Lambda ,\lambda \geq \mu }$, then there is $a\in A$
such that $\pi _{\lambda }^{A}\left( a\right) =a_{\lambda }\ $for all $%
\lambda \in \Lambda $. From $\pi _{\lambda }^{A}\left( \alpha _{g}\left(
a\right) -a\right) =\alpha _{g}^{\lambda }\left( \pi _{\lambda }^{A}\left(
a\right) \right) -\pi _{\lambda }^{A}\left( a\right) =0\ $for all $\lambda
\in \Lambda $, we deduce that $a\in \left( A[\tau _{\Gamma }]\right)
^{\alpha }$, and so $\Psi $ is surjective. Therefore, $\Psi $ is an
isomorphism \bigskip of locally $C^{\ast }$-algebras.
\end{proof}

\begin{remark}
We remark that, in general, the isometric $C^{\ast }$-morphism $\psi
_{\lambda }:\left( \left( A[\tau _{\Gamma }]\right) ^{\alpha }\right)
_{\lambda }$ $\rightarrow \left( A_{\lambda }\right) ^{\alpha ^{\lambda
}},\psi _{\lambda }\left( a+\ker \left( p_{\lambda }|_{\left( A[\tau
_{\Gamma }]\right) ^{\alpha }}\right) \right) =a+\ker p_{\lambda }$ is not
onto.
\end{remark}

By \cite[Theorem 3.2]{K}, the crossed product $A\times _{\alpha }G$ of a $%
C^{\ast }$-algebra $A$ by an action $\alpha $ of a compact abelian group $G$
is a scattered $C^{\ast }$-algebra if and only if $A^{\alpha }$ is a
scattered $C^{\ast }$-algebra. We do not know if this result is true in the
context of locally $C^{\ast }$-algebras, but we can prove the following
results.

\begin{proposition}
\label{fixe} Let $G$ be a compact abelian group and $\alpha $ an action of $%
G $ on a locally $C^{\ast }$-algebra $A[\tau _{\Gamma }]$. If $A\times
_{\alpha }G$ is scattered, then $\left( A[\tau _{\Gamma }]\right) ^{\alpha }$
is scattered.
\end{proposition}

\begin{proof}
If $A\times _{\alpha }G$ is scattered, by Proposition \ref{scatt} and by 
\cite[Theorem 3.2]{K} , $\left( A_{\lambda }\right) ^{\alpha ^{\lambda
}},\lambda \in \Lambda $, are scattered. Since for each $\lambda \in \Lambda
,\left( \left( A[\tau _{\Gamma }]\right) ^{\alpha }\right) _{\lambda }$ is a 
$C^{\ast }$-subalgebra of $\left( A_{\lambda }\right) ^{\alpha ^{\lambda }}$%
, $\left( \left( A[\tau _{\Gamma }]\right) ^{\alpha }\right) _{\lambda }$ is
scattered and so $\left( A[\tau _{\Gamma }]\right) ^{\alpha }$ is scattered.
\end{proof}

\begin{definition}
We say that an inverse limit action $\alpha =\lim\limits_{\leftarrow \lambda
}\alpha ^{\lambda }$ of a locally compact group $G$ on a pro-$C^{\ast }$%
-algebra $A[\tau _{\Gamma }]$ is perfect, if $\{\left( A_{\lambda }\right)
^{\alpha ^{\lambda }};\pi _{\lambda \mu }^{A}|_{\left( A_{\lambda }\right)
^{\alpha ^{\lambda }}}\}_{\lambda ,\mu \in \Lambda ,\lambda \geq \mu }$ is a
perfect inverse system of $C^{\ast }$-algebras.
\end{definition}

\begin{example}
Let $G$\ be a locally compact group and $A[\tau _{\Gamma }]$\ a locally $%
C^{\ast }$-algebra. The action $\delta $\ of $G$\ on the locally $C^{\ast }$%
-algebra $C_{0}\left( G,A\right) $\ of all continuos functions from $G$\ to $%
A$\ vanishing to infinite, given by $\delta _{g}\left( f\right) \left(
t\right) =f\left( tg\right) \ $for all $g,t\in G$, is an inverse limit
action, $\delta =\lim\limits_{\leftarrow \lambda }\delta ^{\lambda }$, where 
$\delta ^{\lambda }$\ is the action of $G$\ on $C_{0}\left( G,A_{\lambda
}\right) ,$\ given by $\delta _{g}^{\lambda }\left( f\right) \left( t\right)
=f\left( tg\right) \ $for all $g,t\in G.$\ Moreover, $\delta $\ is perfect,
since the fixed point algebra of $C_{0}\left( G,A\right) $\ under $\delta $\
is isomorphic with $A.$
\end{example}

Under perfectness of the action, we succeed the inverse statement of
Proposition \ref{fixe}.

\begin{theorem}
Let $G$ be a compact abelian group and $\alpha $ a perfect action of $G$ on
a locally $C^{\ast }$-algebra $A[\tau _{\Gamma }]$. Then the following
statements are equivalent:

\begin{enumerate}
\item $A\times _{\alpha }G$ is scattered;

\item $\left( A[\tau _{\Gamma }]\right) ^{\alpha }$ is scattered.
\end{enumerate}
\end{theorem}

\begin{proof}
$\left( 1\right) \Rightarrow \left( 2\right) .$ It follows from Proposition %
\ref{fixe}.

$\left( 2\right) \Rightarrow \left( 1\right) .$ Since $\left( A[\tau
_{\Gamma }]\right) ^{\alpha }$\ is a scattered locally $C^{\ast }$-algebra,
the factors $\left( \left( A[\tau _{\Gamma }]\right) ^{\alpha }\right)
_{\lambda },$ $\lambda \in \Lambda ,$ in the Arens-Michael decomposition of $%
\left( A[\tau _{\Gamma }]\right) ^{\alpha }$, are scattered. On the other
hand, $\left( A[\tau _{\Gamma }]\right) ^{\alpha
}=\lim\limits_{\leftharpoonup \lambda }\left( A_{\lambda }\right) ^{\alpha
^{\lambda }}$, up to an isomorphism of locally $C^{\ast }$-algebras, and $%
\{\left( A_{\lambda }\right) ^{\alpha ^{\lambda }};\pi _{\lambda \mu
}^{A}|_{\left( A_{\lambda }\right) ^{\alpha ^{\lambda }}}\}_{\lambda ,\mu
\in \Lambda ,\lambda \geq \mu }$ is a perfect inverse system of $C^{\ast }$%
-algebras. Therefore, the fixed point algebras $\left( A_{\lambda }\right)
^{\alpha ^{\lambda }}$ of $A_{\lambda }$ under $\alpha ^{\lambda },\lambda
\in \Lambda ,$ are scattered. Then, by \cite[Theorem 3.2]{K}, the factors $%
A_{\lambda }\times _{\alpha ^{\lambda }}G,\lambda \in \Lambda ,$ in the
Arens-Michael decomposition of $A\times _{\alpha }G$ , are scattered and by
Proposition \ref{scatt}, $A\times _{\alpha }G$ is scattered.
\end{proof}

Let $\alpha =\lim\limits_{\leftarrow \lambda }\alpha ^{\lambda }$ be an
inverse limit action of a locally compact group $G$ on a locally $C^{\ast }$%
-algebra $A[\tau _{\Gamma }]$. For each $g\in G$,$\ $the restriction of $%
\alpha _{g}$ on $b\left( A[\tau _{\Gamma }]\right) $, the $C^{\ast }$%
-algebra of all bounded elements of $A[\tau _{\Gamma }]$, is an automorphism
of $b\left( A[\tau _{\Gamma }]\right) \ $and the map $g\mapsto \alpha
_{g}|_{b\left( A[\tau _{\Gamma }]\right) }$ from $G$ to Aut$\left( b\left(
A[\tau _{\Gamma }]\right) \right) $ is a group morphism. In general, the map 
$g\mapsto \alpha _{g}|_{b\left( A[\tau _{\Gamma }]\right) }$ from $G$ to Aut$%
\left( b\left( A[\tau _{\Gamma }]\right) \right) $ is not an action of $G$
on $b\left( A[\tau _{\Gamma }]\right) $, since for a fixed element $a\in
b(A[\tau _{\Gamma }])$, the map $g\mapsto $ $\alpha _{g}(a)\ $from $G$ to $%
b\left( A[\tau _{\Gamma }]\right) $ is not always continuous. We remark that 
$\beta $ is an action of $G$ on $b\left( A[\tau _{\Gamma }]\right) $ such
that the closed two-sided $\ast $-ideals $\ker p_{\lambda }|_{b(A[\tau
_{\Gamma }])}$, $\lambda \in \Lambda $, are $\beta $ -invariant, then $\beta 
$ extends to an action $g\mapsto \widetilde{\beta }_{g}$ of $G$ on $A[\tau
_{\Gamma }],$ where $\widetilde{\beta }_{g}$ is the extension of the $\ $%
automorphism $\beta _{g}$ of $b\left( A[\tau _{\Gamma }]\right) $ to an $\ $%
automorphism of $A[\tau _{\Gamma }]$.

Suppose that for each $a\in b(A[\tau _{\Gamma }])$, the map $g\mapsto $ $%
\alpha _{g}(a)\ $from $G$ to $b\left( A[\tau _{\Gamma }]\right) $ is
continuous. Then $\alpha |_{b(A[\tau _{\Gamma }])}$ is an action of $G$ on $%
b(A[\tau _{\Gamma }])$ and the closed two-sided $\ast $-ideals $\ker
p_{\lambda }|_{b(A[\tau _{\Gamma }])}$, $\lambda \in \Lambda $, are $\alpha
|_{b(A[\tau _{\Gamma }])}$ -invariant. Therefore, for each $\lambda \in
\Lambda ,\alpha |_{b(A[\tau _{\Gamma }])}$ induces an action of $G$ on $%
b(A[\tau _{\Gamma }])/\ker p_{\lambda }|_{b(A[\tau _{\Gamma }])},$ denoted
by $\alpha |_{b(A[\tau _{\Gamma }])}^{\lambda }$, and the $C^{\ast }$%
-algebras $(b(A[\tau _{\Gamma }])\times _{\alpha |_{b(A[\tau _{\Gamma
}])}}G)/\ker p_{\lambda }|_{b(A[\tau _{\Gamma }])}\times _{\alpha
|_{b(A[\tau _{\Gamma }])}}G$ \ and $b(A[\tau _{\Gamma }])/\ker p_{\lambda
}|_{b(A[\tau _{\Gamma }])}\times _{\alpha |_{b(A[\tau _{\Gamma }])}^{\lambda
}}G$ are isomorphic (see, for example, \cite[IV.3.5.8]{B}).

\begin{lemma}
\label{bounded}Let $\alpha =\lim\limits_{\leftarrow \lambda }\alpha
^{\lambda }$ be an inverse limit action of a locally compact group $G$ on a
locally $C^{\ast }$-algebra $A[\tau _{\Gamma }]$ such that for each $a\in
b(A[\tau _{\Gamma }])$, the map $g\mapsto $ $\alpha _{g}(a)\ $from $G$ to $%
b\left( A[\tau _{\Gamma }]\right) $ is continuous. Then, for each $\lambda
\in \Lambda $,$\ $the $C^{\ast }$-algebras $\left( b(A[\tau _{\Gamma
}])\times _{\alpha |_{b(A[\tau _{\Gamma }])}}G\right) /\ker p_{\lambda
}|_{b(A[\tau _{\Gamma }])}\times _{\alpha |_{b(A[\tau _{\Gamma }])}}G$ and $%
A_{\lambda }\times _{\alpha ^{\lambda }}G\ $are isomorphic.
\end{lemma}

\begin{proof}
Take $\lambda \in \Lambda $. The map $\varphi _{\lambda }:b(A[\tau _{\Gamma
}])/\ker p_{\lambda }|_{b(A[\tau _{\Gamma }])}\rightarrow A[\tau _{\Gamma
}])/\ker p_{\lambda }$, given by 
\begin{equation*}
\varphi _{\lambda }\left( a+\ker p_{\lambda }|_{b(A[\tau _{\Gamma
}])}\right) =a+\ker p_{\lambda }
\end{equation*}%
is a $C^{\ast }$-isomorphism (see, for example, \cite[Theorem 10.24]{F} ).
Moreover, 
\begin{equation*}
\varphi _{\lambda }\left( \left( \alpha |_{b(A[\tau _{\Gamma }])}\right)
_{g}\left( a+\ker p_{\lambda }|_{b(A[\tau _{\Gamma }])}\right) \right)
=\alpha _{g}^{\lambda }\left( \varphi _{\lambda }\left( a+\ker p_{\lambda
}|_{b(A[\tau _{\Gamma }])}\right) \right)
\end{equation*}%
for all $a\in $ $b(A[\tau _{\Gamma }])$ and for all $g\in G$. Thus, there is
a $C^{\ast }$-isomorphism $\Phi _{\lambda }:b(A[\tau _{\Gamma }])/\ker
p_{\lambda }|_{b(A[\tau _{\Gamma }])}\times _{\alpha |_{b(A[\tau _{\Gamma
}])}^{\lambda }}G$ $\rightarrow $ $A_{\lambda }\times _{\alpha ^{\lambda }}G$
such that 
\begin{equation*}
\Phi _{\lambda }\left( \left( a+\ker p_{\lambda }|_{b(A[\tau _{\Gamma
}])}\right) \otimes f\right) =\varphi _{\lambda }\left( a+\ker p_{\lambda
}|_{b(A[\tau _{\Gamma }])}\right) \otimes f
\end{equation*}%
for all $a\in $ $b(A[\tau _{\Gamma }])$ and for all $f\in C_{c}(G)$.
Therefore, the $C^{\ast }$-algebras $A_{\lambda }\times _{\alpha ^{\lambda
}}G$ and $\left( b(A[\tau _{\Gamma }])\times _{\alpha |_{b(A[\tau _{\Gamma
}])}}G\right) /\ker p_{\lambda }|_{b(A[\tau _{\Gamma }])}\times _{\alpha
|_{b(A[\tau _{\Gamma }])}}G\ $are isomorphic.
\end{proof}

\begin{proposition}
\label{inv} Let $\alpha =\lim\limits_{\leftarrow \lambda }\alpha ^{\lambda }$
be an inverse limit action of a locally compact group $G$ on a locally $%
C^{\ast }$-algebra $A[\tau _{\Gamma }]$ such that for each $a\in b(A[\tau
_{\Gamma }])$, the map $g\mapsto $ $\alpha _{g}(a)\ $from $G$ to $b\left(
A[\tau _{\Gamma }]\right) $ is continuous. If $b(A[\tau _{\Gamma }])\times
_{\alpha |_{b(A[\tau _{\Gamma }])}}G$ is scattered, then $A\times _{\alpha
}G $ is scattered.
\end{proposition}

\begin{proof}
It $b(A[\tau _{\Gamma }])\times _{\alpha |_{b(A[\tau _{\Gamma }])}}G$ is
scattered, then for each $\lambda \in \Lambda $, the $C^{\ast }$-algebra $%
\left( b(A[\tau _{\Gamma }])\times _{\alpha |_{b(A[\tau _{\Gamma
}])}}G\right) /\ker p_{\lambda }|_{b(A[\tau _{\Gamma }])}\times _{\alpha
|_{b(A[\tau _{\Gamma }])}}G$ is scattered and according to Lemma \ref%
{bounded}, for each $\lambda \in \Lambda $, the $C^{\ast }$-algebra $%
A_{\lambda }\times _{\alpha ^{\lambda }}G$ is scattered. From this fact, 
\cite[Corollary 1.3.7]{J} and Proposition \ref{scatt}, we deduce that $%
A\times _{\alpha }G$ is scattered.
\end{proof}

\begin{theorem}
Let $G$ be a compact abelian group and let $\alpha $ be an action of $G$ on
a locally $C^{\ast }$-algebra $A[\tau _{\Gamma }]$ such that for each $a\in
b(A[\tau _{\Gamma }])$, the map $g\mapsto $ $\alpha _{g}(a)\ $from $G$ to $%
b\left( A[\tau _{\Gamma }]\right) $ is continuous. If for some $\lambda \in
\Lambda $, the closed two sided $\ast $-ideal $\ker p_{\lambda }|_{b(A[\tau
_{\Gamma }])}$ of $b(A[\tau _{\Gamma }])$ is scattered, then the following
statements are equivalent:

\begin{enumerate}
\item $\left( b(A[\tau _{\Gamma }])\right) ^{\alpha |_{b(A[\tau _{\Gamma
}])}}$ is scattered;

\item $b(A[\tau _{\Gamma }])\times _{\alpha |_{b(A[\tau _{\Gamma }])}}G$ is
scattered;

\item $A\times _{\alpha }G$ is scattered;

\item $\left( A[\tau _{\Gamma }])\right) ^{\alpha }$ is scattered.
\end{enumerate}
\end{theorem}

\begin{proof}
$\left( 1\right) \Leftrightarrow \left( 2\right) .$ By \cite[Theorem 3.2]{K}%
, $\left( b(A[\tau _{\Gamma }])\right) ^{\alpha |_{b(A[\tau _{\Gamma }])}}$
is scattered if and only if $b(A[\tau _{\Gamma }])\times _{\alpha
|_{b(A[\tau _{\Gamma }])}}G$ is scattered.

$(2)$ $\Longrightarrow (3).$ It follows from Proposition \ref{inv}.

$(3)$ $\Longrightarrow (4).$ It follows from Proposition \ref{fixe}.

$(4)$ $\Longrightarrow (1).$ If the closed two sided $\ast $-ideal $\ker
p_{\lambda }|_{b(A[\tau _{\Gamma }])}$ of $b(A[\tau _{\Gamma }])$ is
scattered, then $\ker p_{\lambda }|_{\left( b(A[\tau _{\Gamma }])\right)
^{\alpha |_{b(A[\tau _{\Gamma }])}}}$ is scattered and, by Proposition \ref%
{bond}, $\left( b(A[\tau _{\Gamma }])\right) ^{\alpha |_{b(A[\tau _{\Gamma
}])}}$ is scattered, since $\left( b(A[\tau _{\Gamma }])\right) ^{\alpha
|_{b(A[\tau _{\Gamma }])}}=b\left( \left( A[\tau _{\Gamma }])\right)
^{\alpha }\right) $.
\end{proof}

\begin{acknowledgement}
The author would like to thank Professor M. Haralampidou of the University
of Athens for valuable comments and suggestions during the preparation of
this paper.\textbf{\ }The author was partially supported by the grant of the
Romanian Ministry of Education, CNCS - UEFISCDI, project number
PN-II-ID-PCE-2012-4-0201.
\end{acknowledgement}

\bigskip

\bigskip

\end{document}